\documentclass{scrartcl}

\RequirePackage[colorlinks,citecolor=blue,urlcolor=blue]{hyperref}
\usepackage{amsmath}
\usepackage{amsfonts}
\usepackage{dsfont}
\usepackage[T1]{fontenc}
\usepackage[latin1]{inputenc}
\usepackage{amsthm}
\usepackage{multirow}
\usepackage{hhline}
\usepackage{graphicx}
\usepackage{bm}
\usepackage{listings}
\usepackage[labelfont=bf, figurename=Fig.]{caption}

\usepackage{tikz}
\usetikzlibrary{positioning}
\usetikzlibrary{shapes,snakes}
\usetikzlibrary{shapes,arrows}

\usepackage{booktabs,color,epsfig,url}

\makeatletter
\DeclareOldFontCommand{\rm}{\normalfont\rmfamily}{\mathrm}
\DeclareOldFontCommand{\sf}{\normalfont\sffamily}{\mathsf}
\DeclareOldFontCommand{\tt}{\normalfont\ttfamily}{\mathtt}
\DeclareOldFontCommand{\bf}{\normalfont\bfseries}{\mathbf}
\DeclareOldFontCommand{\it}{\normalfont\itshape}{\mathit}
\DeclareOldFontCommand{\sl}{\normalfont\slshape}{\@nomath\sl}
\DeclareOldFontCommand{\sc}{\normalfont\scshape}{\@nomath\sc}
\makeatother

\newcommand{\E}{\rm{E}}
\newcommand{\x}{\bold{x}}
\newcommand{\D}{{\mathcal{D}}}
\newcommand{\F}{{\mathcal{F}}}
\newcommand{\N}{\mathbb{N}}

\newcommand{\R}{\mathbb{R}}
\newcommand{\Rd}{\mathbb{R}^d}

\newtheorem{theorem}{Theorem}

\newtheorem{lemma}{Lemma}

\newtheorem{definition}{Definition}
\theoremstyle{definition}
\newtheorem{remark}{Remark}

\newcommand{\EXP}{{\rm{E}}}
\newcommand{\PROB}{\Pr}
\allowdisplaybreaks

\begin{document}

\begin{center}
{\LARGE \textbf{
Analysis of the rate of convergence of fully connected deep neural network regression estimates with smooth activation function}}
\footnote{
Running title: {Fully connected deep neural network regression estimates with smooth activation function}}
\vspace{0.5cm}

Sophie Langer\footnote{
Corresponding author. Tel: +49-6151-16-23371} 

{\textit{Fachbereich Mathematik, Technische Universit\"at Darmstadt, \\
Schlossgartenstr. 7, 64289 Darmstadt, Germany,
\\
email: langer@mathematik.tu-darmstadt.de}}

\end{center}
\vspace{0.5cm}

\begin{center}
October 8, 2018
%\today
\end{center}
\vspace{0.5cm}

\noindent
{\textbf{Abstract}}\\
This article contributes to the current statistical theory of deep neural networks (DNNs). 
It was shown that DNNs are able to circumvent the so--called
curse of dimensionality in case that suitable restrictions
on the structure of the regression function hold. In most of those results the tuning parameter 
is the \textit{sparsity} of the network, which describes the number of non-zero weights in the network. 
This constraint seemed to be the key factor for the good rate of convergence results. Recently, the assumption
was disproved. In particular, it was shown that simple fully connected DNNs can achieve the same rate of convergence. Those fully connected DNNs are based on the unbounded ReLU activation function. In this article we extend the results to smooth activation functions, i.e., to the sigmoid activation function. It is shown that estimators based on fully connected DNNs with sigmoid activation function also achieve the minimax rates of convergence (up to $\ln n$-factors). In our result the number of hidden layers is fixed, the number of neurons per layer tends to infinity for sample size tending to infinity and a bound for the weights in the network is given.

\vspace*{0.2cm}


\vspace*{0.2cm}


\section{Introduction}
Deep neural networks (DNNs) have been shown great success in various tasks like pattern recognition and nonparametric regression (see, e.g.,
the monographs \cite{AB09, DGL96, GKKW02, H98, HPK91, R95}). Unfortunately, little is yet known about why 
this method is so successful in practical applications. In particular, there is still a gap between the practical use and the theoretical understanding, which have to be filled to provide a method which is efficient and reliable. This article is inspired to contribute to the current statistical theory of DNNs. The most convenient way to do this is to analyze DNNs in the context of nonparametric regression. 

\subsection{Nonparametric regression}
In nonparametric regression a $\Rd \times \R$--valued random vector
$(\bold{X},Y)$
satisfying 
$\EXP \{Y^2\}<\infty$ is considered. Given a sample
of size $n$ of $(\bold{X},Y)$, i.e., given a data set
\begin{equation*}
\D_n = \left\{(\bold{X}_1,Y_1), \ldots, (\bold{X}_n,Y_n)\right\},
\end{equation*}
where
$(\bold{X},Y)$, $(\bold{X}_1,Y_1)$, \ldots, $(\bold{X}_n,Y_n)$ are i.i.d., the aim is to construct an estimator
\[
m_n(\cdot)=m_n(\cdot, \D_n):\Rd \rightarrow \R
\]
of the so--called regression function $m:\Rd \rightarrow \R$,
$m(\x)=\EXP\{Y|\bold{X}=\x\}$ such that the so--called $L_2$ error
\[
\int |m_n(\x)-m(\x)|^2 {\Pr}_{\bold{X}} (d\x)
\]
is ``small'' (cf., e.g., \cite{GKKW02}
for a systematic introduction to nonparametric regression and
a motivation for the $L_2$ error).

\subsection{Neural networks}
In order to construct such regression estimators with DNNs, the first step is to define a suitable
space of functions $f:\Rd \rightarrow \R$ by using neural networks.
The starting point here is the choice of an activation function $\sigma: \mathbb{R} \to \mathbb{R}$.
Traditionally, so--called squashing functions are chosen as activation
function $\sigma: \mathbb{R} \to \mathbb{R}$, which are nondecreasing
and satisfy $\lim_{x \rightarrow - \infty} \sigma(x)=0$
and
$\lim_{x \rightarrow  \infty} \sigma(x)=1$,
e.g., the so-called sigmoid activation function
\begin{equation}
  \label{inteq4}
\sigma(x)=\frac{1}{1+\exp(-x)}, \quad x \in \R.
\end{equation}
Recently, also unbounded activation functions are used, e.g., the
ReLU activation function $\sigma(x)=\max\{x,0\}$.

The network architecture $(L, \textbf{k})$ depends on a positive integer $L$ called the \textit{number of hidden layers} and a \textit{width vector} $\textbf{k} = (k_1, \ldots, k_{L}) \in \mathbb{N}^{L}$ that describes the number of neurons in the first, second, $\ldots$, $L$-th hidden layer. A feedforward DNN with network architecture $(L, \textbf{k})$ and sigmoid activation function $\sigma$ is a real-valued function defined on $\mathbb{R}^d$ of the form
\begin{equation}\label{inteq1}
f(\x) = \sum_{i=1}^{k_L} c_{1,i}^{(L)}f_i^{(L)}(\x) + c_{1,0}^{(L)}
\end{equation}
for some $c_{1,0}^{(L)}, \ldots, c_{1,k_L}^{(L)} \in \mathbb{R}$ and for $f_i^{(L)}$'s recursively defined by
\begin{equation}
  \label{inteq2}
f_i^{(s)}(\x) = \sigma\left(\sum_{j=1}^{k_{s-1}} c_{i,j}^{(s-1)} f_j^{(s-1)}(\x) + c_{i,0}^{(s-1)} \right)
\end{equation}
for some $c_{i,0}^{(s-1)}, \dots, c_{i, k_{s-1}}^{(s-1)} \in \mathbb{R}$,
$s \in \{2, \dots, L\}$,
and
\begin{equation}
  \label{inteq3}
f_i^{(1)}(\x) = \sigma \left(\sum_{j=1}^d c_{i,j}^{(0)} x^{(j)} + c_{i,0}^{(0)} \right)
\end{equation}
for some $c_{i,0}^{(0)}, \dots, c_{i,d}^{(0)} \in \mathbb{R}$. 
The space of DNNs with 
$L$ hidden layers, $r$ neurons per layer and all coefficients bounded by $\alpha$ is defined by
\begin{align}\label{F}
  \mathcal{F}(L, r, \alpha) = \{ &f \, : \,  \text{$f$ is of the form } \eqref{inteq1}
  \text{ with }
k_1=k_2=\ldots=k_L=r \ \notag\\
&\text{and} \
|c_{i,j}^{(\ell)}| \leq \alpha \ \text{for all} \ i,j,\ell
\}. 
\end{align}
Since the networks of this function space are only defined by its width and depth (and by a bound for the weights in the network) we refer to this function space, as in \cite{YZ19} and \cite{KL20} as \textit{fully connected} DNNs. 

\subsection{Least squares estimator}
A corresponding estimator can then be defined with the principle of least squares. In particular, we choose $L=L_n$ hidden layers, a number $r=r_n$ of neurons per hidden layer and bound $\alpha = \alpha_n$ for all coefficients in the network in dependence to the sample size. The fully connected DNN regression estimator is then defined 
%
%Using these function spaces with some properly chosen
%number $L=L_n$ of hidden layers, number $r=r_n$ of neurons per
%hidden layer and bound $\alpha = \alpha_n$ for all coefficients a neural network regression estimate can be defined
%by using the principle of least squares. 
%To do this, one defines
%the neural network regression estimate 
as the minimizer of the so--called
empirical $L_2$ risk over the function space $\F(L_n,r_n, \alpha_n)$,
which results in
\[
m_n(\cdot)
=
\arg \min_{f \in \F(L_n,r_n, \alpha_n)}
\frac{1}{n} \sum_{i=1}^n |f(\bold{X}_i)-Y_i|^2.
\]
For simplicity we assume here and in the sequel that the minimum above indeed exists. When this is not the case our theoretical results also hold for any estimator which minimizes the above empirical $L_2$ risk up to a small additional term. 
\newline
\subsection{Curse of dimensionality}
In order to judge the quality of such estimators theoretically, usually
the rate of convergence of the $L_2$ error is considered.
It is well-known, that smoothness assumptions on the
regression function are necessary in order to derive non-trivial
results on the rate of convergence
(see, e.g., Theorem 7.2 and Problem 7.2 in 
\cite{DGL96} and
Section 3 in \cite{DW80}).
 For that purpose, we introduce the following definition of $(p,C)$-smoothness.
\begin{definition}
\label{intde2} 
  Let $p=q+s$ for some $q \in \N_0$ and $0< s \leq 1$.
A function $m:\R^d \rightarrow \R$ is called
$(p,C)$-smooth, if for every $\bm{\alpha}=(\alpha_1, \dots, \alpha_d) \in
\N_0^d$
with $\sum_{j=1}^d \alpha_j = q$ the partial derivative
$\partial^q m/(\partial x_1^{\alpha_1}
\dots
\partial x_d^{\alpha_d}
)$
exists and satisfies
\[
\left|
\frac{
\partial^q m
}{
\partial x_1^{\alpha_1}
\dots
\partial x_d^{\alpha_d}
}
(\x)
-
\frac{
\partial^q m
}{
\partial x_1^{\alpha_1}
\dots
\partial x_d^{\alpha_d}
}
(\bold{z})
\right|
\leq
C
\| \x-\bold{z} \|^s
\]
for all $\x,\bold{z} \in \R^d$, where $\Vert\cdot\Vert$ denotes the Euclidean norm.
\end{definition}
For this function space the optimal minimax rate of convergence in nonparametric regression is given by
\begin{align*}
n^{-2p/(2p+d)}
\end{align*}
(see, e.g.,  \cite{Sto82}).
%
%
% showed that the optimal minimax rate of convergence in nonparametric
%regression for $(p,C)$-smooth functions is $n^{-2p/(2p+d)}$. 
This rate suffers from a characteristic feature in case of high-dimensional functions: If $d$ is relatively large compared to $p$, then this rate of convergence can be extremely slow. This phenomenon is well-known as the curse of dimensionality and the only way to circumenvent it is by imposing additional assumptions on the regression function. \cite{Sto85, Sto94} assumed some additive structure on the regression function and showed some optimal minimax rate of convergence independent of the input dimension $d$. Other classes like a so-called single index models, in which
\[
m(\x) = g(a^{\top} \x), \quad \x \in \Rd
\]
is assumed to hold, where $g: \R \rightarrow \R$ is a univariate
function and $a \in \Rd$ is a $d$-dimensional vector were considered in \cite{Ha93, HaSt89,KoXi07,YYR02}.
Related to this is  the so-called projection pursuit, where the regression function
is assumed to be a sum of functions of the above form, i.e.,
\[
m(\x) = \sum_{k=1}^K g_k(a_k^{\top} \x), \quad \x \in \Rd
\]
for $K \in \N$, $g_k: \R \rightarrow \R$ and $a_k \in \Rd$ (see, e.g., \cite{FrSt81}). If we assume that the univariate functions in these postulated structures are
$(p,C)$-smooth, adequately chosen regression estimators can achieve the above univariate rates of convergence up to some logarithmic factor
(cf., e.g., Chapter 22 in \cite{GKKW02}). \cite{HM07} studied the case of a regression function, which satisfies
\[
m(\x)=g\left(\sum_{\ell_1=1}^{L_1}g_{\ell_1}  \left(\sum_{\ell_2=1}^{L_2}g_{\ell_1,\ell_2}\left( \ldots \sum_{\ell_r=1}^{L_r}g_{\ell_1,\ldots, \ell_r}(\x^{\ell_1,\ldots,\ell_r}) \right)\right)\right),
\]
where $g, g_{\ell_1}, \ldots, g_{\ell_1,\ldots, \ell_r}: \R \rightarrow \R$
are
$(p,C)$-smooth univariate functions and $\x^{\ell_1,\ldots,\ell_r}$ are single components of $\x\in\Rd$ (not necessarily different for two different indices $(\ell_1,\ldots,\ell_r)$).
With the use of a penalized least squares estimator, they proved
that in this setting the rate $n^{-2p/(2p+1)}$ can be achieved.
\subsection{Related results for DNNs}
The rate of convergence
of neural networks regression estimators
has been analyzed by 
\cite{Bar91,Bar93, Bar94, BK17, KoKr05,KoKr17, McCaGa94,Sch17}.
For the $L_2$ error of a
single hidden layer neural network, \cite{Bar94} proves a dimensionless rate of $n^{-1/2}$
(up to some logarithmic factor), provided the Fourier transform has a finite first
moment (which basically
requires that the function becomes smoother with increasing
dimension $d$ of $\bold{X}$).
\cite{McCaGa94} showed a rate of $n^{(-2p/(2p+d+5))+\varepsilon}$ for the $L_2$  error of suitably defined single hidden layer neural network estimator for $(p,C)$-smooth functions, but their study was restricted to the use of a certain cosine squasher as the activation function. The rate of convergence of
neural network regression
estimators based on two layer neural networks has been analyzed in
\cite{KoKr05}. Therein, interaction models were studied,
where the regression function satisfies
\[
m(\x)
=
\sum_{I \subseteq \{1, \dots, d\}, |I|=d^*}
m_I(\x_I), \qquad \x \in \Rd
\]
for some $d^* \in \{1, \dots, d\}$ and $m_I:\R^{d^*} \rightarrow \R$
$(I \subseteq \{1, \dots, d\}, |I| \leq d^*)$, where
\[
\x_{\{i_1,\ldots,i_{d^*}\}}=
(x^{(i_1)}, \dots, x^{(i_{d^*})})
\quad
\mbox{for }
1 \leq i_1 < \ldots < i_{d^*} \leq d,
\]
and
in case that all $m_I$ are $(p,C)$-smooth for some $p \leq 1$
it was shown that suitable neural network estimators achieve a rate of convergence of $n^{-2p/(2p+d^*)}$
(up to some logarithmic factor),
which is again a convergence rate independent of $d$.
In \cite{KoKr17}, this result was extended
to so--called $(p,C)$-smooth generalized hierarchical interaction models of
order $d^*$, which are defined as follows:
\begin{definition}
\label{deold}
Let $d \in \N$, $d^* \in \{1, \dots, d\}$ and $m:\Rd \rightarrow \R$.

\noindent
\textbf{a)}
We say that $m$ satisfies a generalized hierarchical interaction model
of order $d^*$ and level $0$, if there exist $a_1, \dots, a_{d^*} \in
\R^{d}$
and
$f:\R^{d^*} \rightarrow \R$
such that
\[
m(\x) = f(a_1^{\top} \x, \dots, a_{d^*}^{\top} \x)
\quad \mbox{for all } x \in \Rd.
\]

\noindent
\textbf{b)}
We say that $m$ satisfies a generalized hierarchical interaction model
of order $d^*$ and level $\ell+1$, if there exist $K \in \N$,
$g_k: \R^{d^*} \rightarrow \R$ $(k \in \{1, \dots, K\})$
and \linebreak
$f_{1,k}, \dots, f_{d^*,k} :\R^{d} \rightarrow \R$ $(k \in \{1, \dots, K\})$
such that $f_{1,k}, \dots, f_{d^*,k}$
$(k \in \{1, \dots, K\})$
satisfy a generalized
hierarchical interaction model
of order $d^*$ and level $\ell$
and
\[
m(\x) = \sum_{k=1}^K g_k \left(
f_{1,k}(\x), \dots, f_{d^*,k}(\x)
\right)
\quad \mbox{for all } \x \in \Rd.
\]

\noindent
\textbf{c)}
We say that the generalized hierarchical interaction model defined above
is $(p,C)$-smooth, if all functions occurring in
its definition are $(p,C)$--smooth according to \autoref{intde2}.
\end{definition}

%For $l> 0$ and $p,t \geq 1$ we define the underlying function space of all $(p,C)$--smooth hierarchical composition models of level $l$ and order $t$ by 
%\begin{align*}
%\mathcal{H}(l,p,t) = \{&h: \Rd \to \R: h(x) = g(f_{1}(x), \dots, f_{t}(x)) \ \text{with } g:\R^t \to \R \ (p,C) \text{--smooth}\\
%&\text{and } f_j \in \mathcal{H}(l-1, p_{l-1,j}, t_{l-1,j}) \ \text{for some } (p_{l-1,j}, t_{l-1,j}) \in (1, \infty) \times (1, \infty)\}
%\end{align*}
%with
%\begin{align*}
%\mathcal{H}(1,p,t) = \{&h: \Rd \to \R: h(x) = g(x_J) \ \text{with } J \in \{1, \dots, d\}^t \\
%&\text{and}\ g:\R^t \to \R \ (p,C) \text{--smooth}\}.
%\end{align*}
%Each function $h \in \mathcal{H}(l,p,t)$ is a composition of several functions $g$, where each $g$ can have a different input dimension $t$ and a different smoothness $p$. Denote by $\mathcal{P}(\mathcal{H}(l,p,t))$ the parameter set of all tupels $(p,t)$, where the function $h$ 

It was shown that for such models suitably defined multilayer
neural networks (in which the number of hidden layers depends
on the level of the generalized interaction model) achieve the rate of convergence  $n^{-2p/(2p+d^*)}$
(up to some logarithmic factor) in case
$p \leq 1$.
\cite{BK17} showed that this result even holds
for $p>1$ provided the squashing function is suitably
chosen. Similiar rate of convergence
results as in \cite{BK17}
have been shown in
\cite{Sch17}
for neural network regression estimates using
the ReLU activation function. Here slightly more general function spaces, which
fulfill some composition assumption, were studied. \cite{KL20} generalized the function space to so-called hierarchical composition models, i.e., functions which fulfill the following definition.
%\newpage
\begin{definition}
\label{de2}
Let $d \in \N$ and $m: \Rd \to \R$.

\noindent
\textbf{a)}
We say that $m$ satisfies a hierarchical composition model of level $0$ with order and smoothness constraint $\mathcal{P}$, if there exists a $K \in \{1, \dots, d\}$ such that
\[
m(\x) = x^{(K)} \quad \mbox{for all } \x \in \Rd.
\]
\noindent
\textbf{b)}
We say that $m$ satisfies a hierarchical composition model
of level $\ell+1$ with order and smoothness constraint $\mathcal{P}$, if there exist $(p,K) \in \mathcal{P}$, $C >0$,  $g: \R^{K} \to \R$ and $f_{1}, \dots, f_{K}: \Rd \to \R$, such that $g$ is $(p,C)$-smooth, $f_{1}, \dots, f_{K}$ satisfy a  hierarchical composition model of level $\ell$ with order and smoothness constraint $\mathcal{P}$ and 
\[m(\x)=g(f_{1}(\x), \dots, f_{K}(\x)) \quad \mbox{for all } \x \in \Rd\]
\end{definition}
\subsection{Fully connected DNNs}
\cite{KL20} showed for simple fully connected DNN regression estimators with ReLU activation function a rate of convergence of $\max_{(p,K) \in \mathcal{P}} n^{-2p/(2p+K)}$. The networks regarded therein are only defined by its width and depth and contrary to \cite{BK17} and \cite{Sch17} no further sparsity constraint is needed. Reversely, this means, that not the number of nonzero weights, but the number of overall weights of the network is restricted. We see two main advantages in restricting a network in this sense: First, the characterization of a network by its width and depth (and therefore by its overall number of weights) implies the ones in terms of the nonzero weights, while it is not true the other way around. An example is given in \autoref{neunet} and \autoref{neunet1} for the network class $\mathcal{F}(2,5)$. 
\begin{figure}[h!]
\begin{minipage}{0.4\textwidth}
\centering
\pagestyle{empty}
\def\layersep{2.5cm}
 \resizebox{1.1\textwidth}{!}{%
\begin{tikzpicture}[shorten >=1pt,->,draw=black, node distance=\layersep, scale=0.8]
\centering
    \tikzstyle{every pin edge}=[<-,shorten <=1pt]
    \tikzstyle{neuron}=[circle,fill=black!25,minimum size=10pt,inner sep=0pt]
    \tikzstyle{input neuron}=[neuron, fill=
    %green!50
    black];
    \tikzstyle{output neuron}=[neuron, fill=
    %red!50
    black];
    \tikzstyle{hidden neuron}=[neuron, fill=black!50
    %blue!50
    ];
    \tikzstyle{annot} = [
    text centered
    ]

    % Draw the input layer nodes
    \foreach \name / \y in {1,...,4}
    % This is the same as writing \foreach \name / \y in {1/1,2/2,3/3,4/4}
        \node[input neuron, pin=left:\footnotesize{$x^{(\y)}$}, 
        xshift=1cm
        ] (I-\name) at (0,-\y) {};

    % Draw the hidden layer nodes
    \foreach \name / \y in {1,...,5}
        \path[
        yshift=0.5cm
        ]
            node[hidden neuron] (H-\name) at (\layersep,-\y cm) {};
            
     % Draw the hidden layer nodes
    \foreach \name / \y in {1,...,5}
        \path[yshift=1cm]
            node[hidden neuron, right of = H-\name] (H2-\name) {};

    % Draw the output layer node
    \node[output neuron,pin={[pin edge={->}]right:\footnotesize{$f(\bold{x})$}}, right of=H2-3, xshift=-0.8cm] (O) {};

    % Connect every node in the input layer with every node in the
    % hidden layer.
    \foreach \source in {1,...,4}
        \foreach \dest in {1,...,5}
            \path (I-\source) edge (H-\dest);
            
     % Connect every node in the hidden layer with every node in the
    % hidden layer.
    \foreach \source in {1,...,5}
        \foreach \dest in {1,...,5}
            \path (H-\source) edge (H2-\dest);

    % Connect every node in the hidden layer with the output layer
    \foreach \source in {1,...,5}
        \path (H2-\source) edge (O);

    % Annotate the layers
    \node[annot,above of =H-1, xshift=1.3cm, node distance=1cm] (hl) {\footnotesize{Hidden layers}};
    \node[annot,above of=I-1, node distance = 1cm] {\footnotesize{Input}};
   \node[annot,above of=O, yshift=1.5cm, node distance = 1cm] {\footnotesize{Output}};
   \node[draw, below of= H-5, yshift=2.2cm, xshift=-0.4cm, rounded corners, minimum size=1cm] (r) {\footnotesize{$\sigma(\bold{c}^t\bold{x}+c_0)$}};
\end{tikzpicture}}
\caption{A fully connected network of the class $\mathcal{F}(2,5)$}
\label{neunet}
\end{minipage}
\hspace{2cm}
\begin{minipage}{0.4\textwidth}
\centering
\pagestyle{empty}
\def\layersep{2.5cm}
 \resizebox{1.1\textwidth}{!}{%
\begin{tikzpicture}[shorten >=1pt,->,draw=black, node distance=\layersep, scale=0.8]
\centering
    \tikzstyle{every pin edge}=[<-,shorten <=1pt]
    \tikzstyle{neuron}=[circle,fill=black!25,minimum size=10pt,inner sep=0pt]
    \tikzstyle{input neuron}=[neuron, fill=
    %green!50
    black];
    \tikzstyle{output neuron}=[neuron, fill=
    %red!50
    black];
    \tikzstyle{hidden neuron}=[neuron, fill=black!50
    %blue!50
    ];
    \tikzstyle{annot} = [
    text centered
    ]

    % Draw the input layer nodes
    \foreach \name / \y in {1,...,4}
    % This is the same as writing \foreach \name / \y in {1/1,2/2,3/3,4/4}
        \node[input neuron, pin=left:\footnotesize{$x^{(\y)}$}, 
        xshift=1cm
        ] (I-\name) at (0,-\y) {};

    % Draw the hidden layer nodes
    \foreach \name / \y in {1,...,5}
        \path[
        yshift=0.5cm
        ]
            node[hidden neuron] (H-\name) at (\layersep,-\y cm) {};
            
     % Draw the hidden layer nodes
    \foreach \name / \y in {1,...,5}
        \path[yshift=1cm]
            node[hidden neuron, right of = H-\name] (H2-\name) {};

    % Draw the output layer node
    \node[output neuron,pin={[pin edge={->}]right:\footnotesize{$f(\bold{x})$}}, right of=H2-3, xshift=-0.8cm] (O) {};

    % Connect every node in the input layer with every node in the
    % hidden layer.
%    \foreach \source in {1,...,4}
%        \foreach \dest in {1,...,5}
%            \path (I-\source) edge (H-\dest);
            
            \path (I-1) edge (H-2);
            \path (I-1) edge (H-3);
            \path (I-2) edge (H-4);
            \path (I-2) edge (H-1);
            \path (I-3) edge (H-5);
            \path (I-4) edge (H-5);
            \path (I-4) edge (H-3);
            \path (I-4) edge (H-1);
            
     % Connect every node in the hidden layer with every node in the
    % hidden layer.
%    \foreach \source in {2,...,5}
%        \foreach \dest in {1,...,4}
%            \path (H-\source) edge (H2-\dest);

\path (H-1) edge (H2-5);
\path (H-2) edge (H2-1);
\path (H-2) edge (H2-3);
\path (H-3) edge (H2-3);
\path (H-4) edge (H2-1);
\path (H-4) edge (H2-5);
\path (H-5) edge (H2-4);

    % Connect every node in the hidden layer with the output layer
    \foreach \source in {1,...,5}
        \path (H2-\source) edge (O);

    % Annotate the layers
    \node[annot,above of =H-1, xshift=1.3cm, node distance=1cm] (hl) {\footnotesize{Hidden layers}};
    \node[annot,above of=I-1, node distance = 1cm] {\footnotesize{Input}};
   \node[annot,above of=O, yshift=1.5cm, node distance = 1cm] {\footnotesize{Output}};
   \node[draw, below of= H-5, yshift=2.2cm, xshift=-0.4cm, rounded corners, minimum size=1cm] (r) {\footnotesize{$\sigma(\bold{c}^t\bold{x}+c_0)$}};
\end{tikzpicture}}
\caption{A sparsely connected network of the class $\mathcal{F}(2,5)$}
\label{neunet1}
\end{minipage}
\end{figure}

Here we see that both, sparsely connected and fully connected networks, are contained in the network class $\mathcal{F}(2,5)$, while a network with full connectivity (between neurons of consecutive layers) as in \autoref{neunet} is not contained in a network class where the number of nonzero weights is restricted by $20$. Second, the easy topology of the networks enables us an easy and fast implementation of a corresponding estimator. For instance, as shown in Listing \ref{lst:e1}, we can easily implement a least squares DNN regression estimator with the help of Python's packages \texttt{tensorflow} and \texttt{keras}. Remark that this example already uses the sigmoid activation function which fits to the theoretical results of this article. 
\begin{lstlisting}[language=Python, caption={Python code for fitting of fully connected neural networks to data $x_{learn}$ and $y_{learn}$}, label={lst:e1}]

 model = Sequential()
 model.add(Dense(d, activation="sigmoid", input_shape=(d,)))             
 for i in np.arange(L):
    model.add(Dense(K, activation="sigmoid"))
 model.add(Dense(1))
 model.compile(optimizer="adam",
                  loss="mean_squared_error")
 model.fit(x=x_learn,y=y_learn)
   
\end{lstlisting}

\subsection{Main result in this article}
\cite{KL20} analyzes networks with ReLU activation function. We question ourselves if we can show the same rate of convergence for fully connected DNN regression estimators with smooth activation function. In this article we show that this is the case. In particular, we show that we derive a similar rate of convergence as in \cite{BK17, Sch17, KL20} for simple fully connected DNNs with sigmoid activation function. In these networks
the number of  hidden layer is fixed, the number of neurons per
layer tends to infinity for sample size tending to infinity and a bound for the weights in the network is given. In the proofs the approximation results presented in \cite{KL19} are essential.
\subsection{Notation and Outline}
Throughout the paper, the following notation is used:
The sets of natural numbers, natural numbers including $0$ and real numbers
are denoted by $\N$, $\N_0$ and $\R$, respectively. For $z \in \R$, we denote
the smallest integer greater than or equal to $z$ by
$\lceil z \rceil$, and set \newline $z_+=\max\{z,0\}$. Vectors are denoted by bold letters, e.g., $\bold{x} = (x^{(1)}, \dots, x^{(d)})^T$. We define
$\bold{1}=(1, \dots, 1)^T$ and $\bold{0} = (0, \dots, 0)^T$. A $d$-dimensional multi-index is a $d$-dimensional vector 
$\bold{j} = (j^{(1)}, \dots, j^{(d)})^T \in \N_0^d$. As usual, we define
\begin{align*}
&\|\bold{j}\|_1 = j^{(1)} + \dots + j^{(d)}, \quad \bold{x}^{\bold{j}} = (x^{(1)})^{j^{(1)}} \cdots (x^{(d)})^{j^{(d)}}, \\
&\bold{j}! = j^{(1)}! \cdots j^{(d)}!, \quad \partial^{\bold{j}} =\frac{\partial^{j^{(1)}}}{\partial(x^{(1)})^{j^{(1)}}} \cdots \frac{\partial^{j^{(d)}}}{\partial(x^{(d)})^{j^{(d)}}}.
\end{align*}
Let $D \subseteq \R^d$ and let $f:\R^d \rightarrow \R$ be a real-valued
function defined on $\R^d$.
We write $\x = \arg \min_{\bold{z} \in D} f(\bold{z})$ if
$\min_{\bold{z} \in \D} f(\bold{z})$ exists and if
$x$ satisfies
$x \in D$ and $f(\x) = \min_{\bold{z} \in \D} f(\bold{z})$.
The Euclidean and the supremum norms of $\x \in \Rd$
are denoted by $\|\x\|$ and $\|\x\|_\infty$, respectively.
For $f:\R^d \rightarrow \R$
\[
\|f\|_\infty = \sup_{\x \in \R^d} |f(\x)|
\]
is its supremum norm, and the supremum norm of $f$
on a set $A \subseteq \R^d$ is denoted by
\[
\|f\|_{\infty,A} = \sup_{\x \in A} |f(\x)|.
\]
Furthermore we define $\|\cdot\|_{C^q(A)}$ of the smooth function space $C^q(A)$ by
\[
\|f\|_{C^q(A)} := \max\left\{\|\partial^{\bold{j}} f\|_{\infty, A}: \|\bold{j}\|_1 \leq q, \bold{j} \in \N^d\right\}
\]
for any $f \in C^q(A)$.\\
\\
Let $A \subseteq \R^d$, let $\mathcal{F}$ be a set of functions $f: \R^d \to \R$ and let $\epsilon > 0$. A finite collection $f_1, \dots, f_N: \R^d \to \R$ is called an $\epsilon- \Vert \cdot \Vert_{\infty,A}-$ cover of $\mathcal{F}$ if for any $f \in \mathcal{F}$ there exists $i \in \{1, \dots, N\}$ such that
\[
\Vert f - f_i \Vert_{\infty, A} = \sup_{\x \in A} |f(\x) - f_i(\x)| < \epsilon.
\]
The $\epsilon- \Vert \cdot \Vert_{\infty, A}-$ covering number of $\mathcal{F}$ is the size $N$ of the smallest $\epsilon - \Vert \cdot \Vert_{\infty, A}-$ cover of $\mathcal{F}$ and is denoted by $\mathcal{N}(\epsilon, \mathcal{F}, \Vert \cdot \Vert_{\infty, A})$.
We define the truncation operatore $T_{\beta}$ with level $\beta > 0$ as
\begin{equation*}
T_{\beta}u =
\begin{cases}
u \quad &\text{if} \quad |u| \leq \beta\\
\beta  {\rm sign}(u) \quad &\text{otherwise}.
\end{cases}
\end{equation*}

The main result is presented in Section \ref{se2}. Section \ref{se3} deals with a result concerning the approximation of a hierarchical composition model by neural networks.  Section \ref{se4} contains the proof of the main result.
%In Section \ref{se3}
%the finite sample size behavior of these estimates is analyzed
%by applying the estimates to simulated data.

\section{Main result}
\label{se2}
%As already mentioned above, the only possible way to avoid the so--called curse of dimensionality is to restrict 
%the underlying function class. We therefore consider functions, which fulfill the following definition:
%\begin{definition}
%\label{de2}
%Let $d \in \N$ and $m: \Rd \to \R$.
%
%\noindent
%\textbf{a)}
%We say that $m$ satisfies a hierarchical composition model of level $0$, if there exists a $K \in \{1, \dots, d\}$ such that
%\[
%m(x) = x^{(K)} \quad \mbox{for all } x = (x^{(1)}, \dots, x^{(d)})^{\top} \in \Rd.
%\]
%\noindent
%\textbf{b)}
%We say that $m$ satisfies a hierarchical composition model
%of level $\ell+1$, if there exist $K \in \N$, $g: \R^{K} \to \R$ and $f_{1}, \dots, f_{K}: \Rd \to \R$, such that $f_{1}, \dots, f_{K}$ satisfy a  hierarchical composition model of level $\ell$ and 
%\[m(x)=g(f_{1}(x), \dots, f_{K}(x)) \quad \mbox{for all } x \in \Rd.\]
%\noindent
%\textbf{c)}
%We say that a hierarchical composition model satisfies the
%smoothness and order constraint $\P$, where $\P$ is a subset
%of $(0,\infty) \times \N$, if in its definition all
%functions $g$ occuring in part b) satisfy $g:\R^K \rightarrow \R$
%and $g$ $(p,C)$--smooth for some $(p,K) \in \P$ and $C>0$.
%\end{definition}
For $\ell=1$ and some order and smoothness constraint $\mathcal{P} \subseteq (0,\infty) \times \N$ we define our space of hierarchical composition models by
\begin{align*}
\mathcal{H}(1, \mathcal{P}) = \{&h: \R^{d} \to \R: h(\x) = g(x^{(\pi(1))}, \dots, x^{(\pi(K_1^{(1)})}), \text{where} \notag \\
 & g:\R^{K_1^{(1)}} \to \R \ \text{is} \  (p_1^{(1)}, C) \ \text{--smooth} \ \text{for some} \ (p_1^{(1)}, K_1^{(1)}) \in \mathcal{P} \notag \\
 & \text{and} \ \pi: \{1, \dots, K_1^{(1)}\} \to \{1, \dots, d\}\}.
\end{align*}
For $\ell > 1$, it recursively becomes
\begin{align*}
\mathcal{H}(\ell, \mathcal{P}) := \{&h: \R^{d} \to \R: h(\x) = g(f_1(\x), \dots, f_{K_1^{(\ell)}}(\x)), \text{where} \notag\\
& g:\R^{K_1^{(\ell)}} \to \R \ \text{is} \ (p_1^{(\ell)}, C) \text{--smooth} \ \text{for some} \ (p_1^{(\ell)}, K_1^{(\ell)}) \in \mathcal{P} \notag \\
& \text{and} \ f_i \in \mathcal{H}(\ell-1, \mathcal{P})\}.
\end{align*}

 In practice, it is conceivable, that there exist input--output--relationships, which can be described by a regression function contained in $\mathcal{H}(\ell,\mathcal{P})$. Particulary, our assumption is motivated by applications in connection with complex technical systems, which are constructed in a modular form. Here each modular part can be again a complex system, which also explains the recursiv construction in \autoref{de2}.
With regard to other function classes studied in the literature this function class generalizes previous results, as the function class of \cite{BK17} (see \autoref{deold}) forms some special case of $\mathcal{H}(\ell,\mathcal{P})$ in form of an alternation between summation and composition. Compared to the function class studied in \cite{Sch17}, our definition forms a slight generalization, since we allow different smoothness and order constraints within the same level in the composition. 
%After a proper function class is chosen, our main result requires some further constraints regarding the activation function in our neural networks. 
%In particular, we study neural networks with so--called $2$--admissible activation functions, i.e., functions which fulfill the following properties:
%\begin{definition}
%  \label{se2de1}
%  A function $\sigma : \mathbb{R} \to [0, 1]$ is called $2$-admissible,
%  if it is nondecreasing and Lipschitz continous and if, in addition,
%  the following three conditions are satisfied:
%\begin{itemize}
%\item[\rm{(i)}] The function $\sigma$ is $3$ times continously differentiable with bounded derivatives.
%\item[\rm{(ii)}] A point $t_{\sigma} \in \mathbb{R}$ exists, where all derivatives up to the order $2$ of $\sigma$ are different from zero.
%\item[\rm{(iii)}] If $y > 0$, the relation $|\sigma(y) - 1| \leq 1/y $ holds. If $y < 0$, the relation $|\sigma(y)| \leq 1/|y|$ holds. \label{3}
%\end{itemize}
%\end{definition}
%t is easy to see that, e.g., the logistic squasher \eqref{inteq4}
%is $2$--admissible. (cf., e.g., \cite{BK17}).
%\newline
We can now state the main result.
\begin{theorem}
  \label{th1}
 Let $(\bold{X},Y), (\bold{X}_1, Y_1), \dots, (\bold{X}_n, Y_n)$ be independent and identically distributed random variables with values in $\Rd \times \R$ such that $\rm{supp}(\bold{X})$ is bounded and
  \begin{equation*}
  \E\left\{ \exp(c_1 Y^2) \right\} < \infty
  \end{equation*}
  for some constant $c_1 > 0$. Let the corresponding regression function $m$ be contained in the class $\mathcal{H}(\ell, \mathcal{P})$ for some $\ell \in \N$ and $\mathcal{P} \subseteq [1,\infty) \times \N$. Each function $g$ in the definition of $m$ can be of different smoothness $p_g=q_g+s_g$ ($q_g \in \N_0$ and $s_g \in (0,1]$) and of different input dimension $K_g$, where $(p_g,K_g) \in \mathcal{P}$. Denote by $K_{max}$ the maximal input dimension and by $p_{\max}$ the maximal smoothness of one of the functions $g$. Assume that for each $g$ all partial derivatives of order less than or equal to $q_g$ are bounded, i.e., 
   \begin{equation*}
  \|g\|_{C^{q_g}(\R^d)} \leq c_{2}
  \end{equation*}
  for some constant $c_2 >0$. Let each $g$ be Lipschitz continous with Lipschitz constant $C_{Lip} \geq 1$. 
%  Let
%  \begin{align*}
%  (\bar{p},\bar{K}) \in \mathcal{P} \ \text{such that} \ (\bar{p},\bar{K}) = \arg\min_{(p,K) \in \mathcal{P}} \frac{p}{K}.
%  \end{align*}
  Set 
  \begin{itemize}
  \item[{\rm(i)}] $L_n =\ell\left(8+\lceil\log_2(\max\{K_{\max}, p_{\max}+1\})\rceil\right)$
\item[{\rm(ii)}] $r_n = 2^{K_{\max}} \tilde{N}_1 29 \binom{K_{\max}+p_{\max}}{p_{\max}}K_{\max}^2 p_{\max} \max_{(p, K) \in \mathcal{P}}  n^{K/(2(2p+K))}$
\item[{\rm(iii)}] $\alpha_n = n^{c_3}$
\end{itemize}
  with $c_{3} > 0$ sufficiently large.
%  Choose
%  \begin{align*}
%  (\bar{p},\bar{K}) = \arg \max_{(p,K) \in \mathcal{P}} n^{-\frac{2p}{2p+K}}
%  \end{align*}
Let $\sigma: \mathbb{R} \to [0,1]$ be the sigmoid activation function $1/(1+\exp(-x))$. Let $\tilde{m}_n$ be the least squares estimator defined by 
  \begin{equation*}
  \tilde{m}_n (\cdot) = \arg \min_{h \in \mathcal{F}(L_{n},r_{n}, \alpha_n)} \frac{1}{n} \sum_{i=1}^n |Y_i - h(\bold{X}_i)|^2
  \end{equation*}
  and define $m_n = T_{c_4 \ln n} \tilde{m}_n$ for some $c_4 >0$ sufficiently large. Then 
  \begin{equation*}
  \EXP \int |m_n(\x) - m(\x)|^2 {\PROB}_{\bold{X}}(d\x) \leq c_5 (\ln n)^3 \max_{(p,K) \in \mathcal{P}} n^{-\frac{2p}{2p+K}}
  \end{equation*}  
  holds for sufficiently large $n$. 
\end{theorem}

\begin{remark}
\autoref{th1} shows, that the $L_2$ errors of least squares neural network regression estimators based on a set of fully connected DNNs with a fixed number of layers (corresponding to a hierarchical composition model of given level $\ell$ and given smoothness and order constraint $\mathcal{P}$) achieves a rate of convergence $\max_{(p,K) \in \mathcal{P}} n^{-2p/(2p+K)}$ (up to some logarithmic factor), which does not depend on $d$ and which does therefore circumvent the so-called \textit{curse of dimensionality}. 
\end{remark}

\begin{remark}
Due to the fact that some parameters in the definition of the estimator in \autoref{th1} are normally unknown in practice, they have to be chosen in a data--dependent way. Out of a set of different numbers of hidden layers and neurons per layer the best estimator is then chosen adaptively. Several possible methods and their effects can be found in \cite{GKKW02}.
\end{remark}
%
%\begin{remark}
%Similar rate of convergence results can be shown for fully connected neural networks with ReLU-activation function. In contrast to the results in \cite{Sch17} it is here possible to specify the topology of the network completely. In addition our results hold for a slightly more general class of regression function in form of the above hierarchical composition model.
%\end{remark}

%\section{Proofs}
%\label{se4}

\section{Approximation of hierarchical composition models by DNNs}
\label{se3}
The aim of this section is to prove a result concerning the approximation of hierarchical composition models with smoothness and order constraint $\mathcal{P} \subseteq [1, \infty) \times \N$ by DNNs. In order to formulate this result, we observe in a first step, 
%
%that the class $\mathcal{H}(1, \mathcal{P})$ is a set of $(p_1^{(1)},C)$--smooth functions with $K_1^{(1)}$--variate input, such that $(p_1^{(1)},K_1^{(1)}) \in \mathcal{P}$.
%In comparison, the class $\mathcal{H}(l, \mathcal{P})$ for $l >1$ is also a set of $(p_1^{(l)},C)$--smooth function with $K_1^{(l)}$--variate input, where $(p_1^{(l)},K_1^{(l)}) \in \mathcal{P}$, but where each input forms again a $(p_i^{(l-1)}, C)$--smooth function with $K_i^{(l-1)}$--variate input $(i=1, \dots, K_1^{(l)})$, which orginates from $\mathcal{H}(l-1, \mathcal{P})$ per definition. This leads to the fact, 
that one has to compute different hierarchical composition models of some level $i$ $(i\in \{1, \dots, \ell-1\})$ to compute a function $h_1^{(\ell)} \in \mathcal{H}(\ell, \mathcal{P})$. Let $\tilde{N}_i$ denote the number of hierarchical composition models of level $i$, needed to compute $h_1^{(\ell)}$. 
We denote in the following by
\begin{align}
\label{h}
h_j^{(i)}: \R^{d} \to \R 
%\quad ((p_j^{(i)}, K_j^{(i)}) \in \mathcal{P})
\end{align}
the $j$--th hierarchical composition model of some level $i$ ($j \in \{1, \ldots, \tilde{N}_i\}, i \in \{1, \ldots, \ell\}$), that applies a $(p_j^{(i)}, C)$--smooth function $g_j^{(i)}: \R^{K_j^{(i)}} \to \R$ with $p_j^{(i)} = q_j^{(i)} + s_j^{(i)}$, \linebreak $q_j^{(i)} \in \N_0$ and $s_j^{(i)} \in (0,1]$, where $(p_j^{(i)}, K_j^{(i)}) \in \mathcal{P}$.
 The computation of $h_1^{(\ell)}(x)$ can then be recursively described as follows:
    \begin{equation}\label{hji}
  h_j^{(i)}(\x) =  g_{j}^{(i)}\left(h^{(i-1)}_{\sum_{t=1}^{j-1} K_t^{(i)}+1}(\x), \dots, h^{(i-1)}_{\sum_{t=1}^j K_t^{(i)}}(\x) \right)
  \end{equation}
for $j \in \{1, \dots, \tilde{N}_i\}$ and $i \in \{2, \dots, \ell\}$
  and
    \begin{equation}\label{hj1}
  h_j^{(1)}(\x) = g_j^{(1)}\left(x^{\left(\pi(\sum_{t=1}^{j-1} K_t^{(1)}+1)\right)}, \dots, x^{\left(\pi(\sum_{t=1}^{j} K_t^{(1)})\right)}\right)
  \end{equation}
  for some function $\pi: \{1, \dots, \tilde{N}_1\} \to \{1, \dots, d\}$. 
  %Then for some $i=l, l-1, \dots, 1$ the recursion 
  Furthermore for $i \in \{1, \dots, \ell-1\}$ the recursion
\begin{align}
\label{N}
\tilde{N}_l = 1 \ \text{and} \ \tilde{N}_{i} = \sum_{j=1}^{\tilde{N}_{i+1}} K_j^{(i+1)} 
\end{align}
holds.

  \begin{figure}
  \centering
  \small{
  \begin{tikzpicture}[
  level/.style={rectangle = 4pt, draw, text centered, anchor=north, text=black},
  input/.style={rounded corners=7pt, draw, rounded corners=1mm, text centered, anchor=north, text=black},
  level distance=1cm
  ] 
\node (H1l) [level] {$g_1^{(2)}$}
[level distance = 0.5cm]
%			child{
		[sibling distance = 3cm]
%              node (g1) [input] {\scriptsize $g_{1}^{(2)}$}
              [level distance = 1cm]
	            child{
	                node (H1l1) [level] {\scriptsize $g_1^{(1)}$}
	                [level distance = 0.5cm]
%	                 child{
%	               		node (Bla) [input] {\scriptsize $g_{1}^{(1)}$}
	               		[sibling distance = 1.2cm, level distance = 1cm]
	                    child{
	                        node (K1) [level] {\scriptsize $x^{(\pi(1))}$}
	                        }
%	                        child{
%	                            node (K2) [level] {\scriptsize $x^{(\pi(2))}$}
%	                           }
	                                }
          child{
                node (H2l1) [level] {\scriptsize $g_2^{(1)}$}
                [level distance = 0.5cm]
%                child{
%               		node (g2) [input] {\scriptsize $g_{2}^{(1)}$}
                [sibling distance=1.2cm, level distance = 1cm]
                    child{
                        node (K3) [level] {\scriptsize $x^{(\pi(2))}$}
                        }
                        child{
                            node (K4) [level] {\scriptsize $x^{(\pi(3))}$}
                           }
%                            child{
%                            node (K4) [level] {\scriptsize $x^{(\pi(5))}$}
%                           }   
                                }                                           
           child{
                node (HK) [level] {\scriptsize $g_{3}^{(1)}$}
                [level distance = 0.5cm]
%                child{
%               		node (Bla) [input] {\scriptsize $g_{3}^{(1)}$}
               		[sibling distance = 1.2cm, level distance = 1cm]
                    	child{
                        node (H1l2) [level] {\scriptsize $x^{(\pi(4))}$}
                        }
                        child{
                            node (State04) [level] {\scriptsize $x^{(\pi(5))}$}
                           }
                                                   child{
                            node (State04) [level] {\scriptsize $x^{(\pi(6))}$}
                           }
                                }
%                                }                                                                     
                  ; 
    %              \path (H1l) -- (H1l2) node [midway] {$\dots$};
  \end{tikzpicture}}
  \caption{Illustration of a hierarchical composition model of the class $\mathcal{H}(2, \mathcal{P})$ with the structure $h_1^{(2)}(\x) = g_1^{(2)}(h_1^{(1)}(\x), h_2^{(1)}(\x), h_3^{(1)}(\x))$, $h_1^{(1)}(\x) = g_1^{(1)}(x^{(\pi(1))})$, $h_2^{(1)}(\x)=g_2^{(1)}(x^{(\pi(2))}, x^{(\pi(3)})$ and $h_3^{(1)}(\x) = g_3^{(1)}(x^{(\pi(4))}, x^{(\pi(5))}, x^{(\pi(6))})$, defined as in \eqref{hji} and \eqref{hj1}}
\label{h1}
\end{figure}
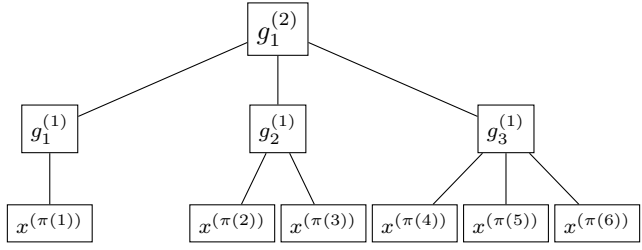
\noindent
The exemplary structure of a function $h_1^{(2)} \in \mathcal{H}(2, \mathcal{P})$ is illustrated in \hyperref[h1]{Fig.\ref*{h1} }. 
Here one can get a perception of how the hierarchical composition models of different levels are stacked on top of each other. The approximation result of such a function $h_1^{(\ell)}$ by a DNN is summarized in the following theorem:

\begin{theorem}
  \label{th3}
 Let $m: \mathbb{R}^d \to \mathbb{R}$ be contained in the class $\mathcal{H}(\ell, \mathcal{P})$ for some $\ell \in \N$ and $\mathcal{P} \subseteq [1,\infty) \times \N$.  Let $\tilde{N}_i$ be defined as in \eqref{N}. Each $m$ consists of different functions $h_j^{(i)}$ $(j \in \{1, \ldots, \tilde{N}_i\},$ 
 $ i\in \{1, \dots, \ell\})$ defined as in \eqref{h}, \eqref{hji} and \eqref{hj1}. Assume that the corrsponding functions $g_j^{(i)}$ are Lipschitz continuous with Lipschitz constant $C_{Lip} \geq 1$ and satisfy
  \begin{equation*}
  \|g_j^{(i)}\|_{C^{q_j^{(i)}}([-a,a]^d)} \leq c_{6}
  \end{equation*}
  for some constant $c_6 >0$. Denote by $K_{max} = \max_{i,j} K_j^{(i)}$ the maximal input dimension and by $p_{max} = \max_{i,j} p_j^{(i)} $
 the maximal smoothness of the functions $g_j^{(i)}$. Then, for any $a \geq 1$, $M_{j,i} \in \mathbb{N}$ sufficiently large (each independent of the size of $a$, but 
  $\min_{j,i} M_{j,i}^{2p_j^{(i)}} > c_{7} \max\{\ell-1 (K_{\max} C_{Lip})^{\ell-2}a^{5p_{\max}+3}, 2^{K_{\max}}, 12K_{\max}\} $ must hold for some constants $c_{7}>0$ sufficiently large) and any 
 \begin{itemize}
\item[{\rm(i)}] $L = \ell\left(8+\lceil\log_2(\max\{K_{\max}, p_{\max}+1\})\rceil\right)$
\item[{\rm(ii)}] $r =\max_{i \in \{1, \dots, \ell\}} \sum_{j=1}^{\tilde{N}_i} 29 \binom{K_j^{(i)}+q_j^{(i)}}{q_j^{(i)}} (K_j^{(i)})^2 (q_{j}^{(i)}+1) M_{j,i}^{K_j^{(i)}}$
\item[{\rm(iii)}] $\alpha = c_{8} a^{24} e^{12\times 2^{2(K_{\max}+1)+1}aK_{\max}}\max_{j,i} M_{j,i}^{20p_{\max}+4K_{\max}+20}$
\end{itemize} 
a neural network $t \in \mathcal{F}\left(L, r, \alpha\right)$ exists such that
\begin{equation*}
\|t-m\|_{\infty, [-a,a]^d} \leq c_{9} a^{5p_{\max}+3} \max_{j,i}M_{j,i}^{-2p_{j}^{(i)}}.
\end{equation*}
\end{theorem}

  In the proof of \autoref{th3} we will need the following auxiliary results. 
  
 \begin{lemma}
  \label{le1}
Let $\sigma: \R \to [0,1]$ be the sigmoid activation function $\sigma(x) = 1/(1+\exp(-x))$. Let $R \geq 1$ and $a>0$. 
Then
\[
f_{id}(x)
=
4R
\left(
\sigma \left(
\frac{x}{R} 
\right)
-
1
\right)
\in \F(1,1,4R)
\]
satisfies for any $x \in [-a,a]$:
\[
| f_{id}(x)-x|
\leq
2\| \sigma^{\prime \prime}\|_{\infty}\frac{a^2}{R}.
\]
\end{lemma}
%%\noindent
%%{\bf b)}
%%Assume that $\sigma$ is three times continuously differentiable and
%%let
%%$t_{\sigma,sq} \in \R$ be such that $\sigma^{\prime \prime}(t_{\sigma,sq}) \neq
%%0$.
%%Then
%%\[
%%f_{sq}(x)
%%=
%%\frac{R^2}{\sigma^{\prime \prime}(t_{\sigma,sq})}
%%\cdot
%%\left(
%%\sigma \left(
%%\frac{2x}{R} + t_{\sigma,sq}
%%\right)
%%-
%%2 \cdot
%%\sigma \left(
%%\frac{x}{R} + t_{\sigma,sq}
%%\right)
%%+
%%\sigma(t_{\sigma,sq})
%%\right)
%%\in \F(1,2,c_{11} \cdot R^2)
%%\]
%%satisfies for any $x \in [-a,a]$:
%%\[
%%| f_{sq}(x)-x^2|
%%\leq
%%\frac{5 \cdot 
%%\| \sigma^{\prime \prime \prime}\|_{\infty} \cdot a^3
%%}{3 \cdot |\sigma^{\prime \prime}(t_{\sigma,sq})|
%%}
%%\cdot
%%\frac{1}{R}.
%%\]
%\end{lemma}
%
\begin{proof}[\rm{\textbf{Proof of \autoref{le1}.}}]
   The result follows in a straightforward way from the proof of
   Theorem 2 in \cite{ScTs98}, cf., e.g.,
   Lemma 1 in \cite{KL20a}.
   \end{proof}
%  
%  a function, let $R \geq 1$ and $a >0$. Assume that $\sigma$ is two times continuously differentiable and let $t_{\sigma} \in \R$ such that $\sigma'(t_{\sigma}) \neq 0$. Then for suitable chosen $\lambda_k$ and $\beta_k$ 
%\begin{equation*}
%  f_{id}(x) = \sum_{k=1}^{2}
%  \gamma_k \cdot \sigma\left(\beta_k \cdot x + t_\sigma \right)
%\end{equation*}
%satisfies for any $x \in [-a,a]$:
%\begin{equation*}
%|f_{id}(x) - x| \leq c_{12} \cdot \frac{a^{2}}{R},
%\end{equation*}
%and the coefficients of this network satisfy
%\[
%|\gamma_k| \leq c_{13} \cdot R^2 \quad \mbox{and} \quad
%|\beta_k| \leq \frac{1}{R} 
%\]
%for $k \in \{1, 2\}$, where $c_{12}$ and $c_{13}$ depend not on $a$ and $R$.  
  
  \begin{lemma}
  \label{le2}
  Let $1 \leq a < \infty$. Let $p=q+s$ for some $q \in \mathbb{N}_0$ and $s \in (0,1]$,
  let $C \geq 1$. Let $m: \mathbb{R}^d \to \mathbb{R}$ be a $(p, C)$-smooth function, which satisfies
  \begin{equation*}
 \|m\|_{C^q([-2a,2a]^d)} \leq c_{10}.
  \end{equation*}
  for some constant $c_{10} > 0$.
  Let $\sigma : \mathbb{R} \to [0, 1]$ be the sigmoid activation function \linebreak $1/(1+\exp(-x))$. Then, for any
  $M \in \mathbb{N}$ sufficiently large (independent of the size of $a$, but
  $c_{11} \max\{a^{5p+3}, 2^d, 12d\} \leq M^{2p}$ must hold for some constant $c_{11} > 0$), a neural network $t \in\mathcal{F}\left(L, r, \alpha\right)$ with
\begin{itemize}
\item[{\rm(i)}] $L \geq 8+\lceil\log_2(\max\{d, q+1\})\rceil$
\item[{\rm(ii)}] $r =29 \binom{d+q}{d} d^2 (q+1) M^d$
\item[{\rm(iii)}] $\alpha = c_{12}\left(\max\left\{a, \|f\|_{C^q([-a,a]^d)}\right\}\right)^{12} e^{6\times 2^{2(d+1)+1}ad} M^{10p+2d+10}$
\end{itemize}
exists such that
\begin{equation*}
\|t-m\|_{\infty, [-a,a]^d} \leq c_{13} a^{5q+3} M^{-2p}.
\end{equation*}
\end{lemma}
  \begin{proof}[\rm{\textbf{Proof of \autoref{le2}.}}]
  For $L=8+\lceil\log_2(\max\{d, q+1\})\rceil$ the proof follows directly from Theorem 1 in \cite{KL19}, where we use that
  \begin{align*}
  &2^d\left(\max\left\{\left(\binom{d+q}{d}+d\right) M^d(2+2d)+d, 4(q+1)\binom{d+q}{d}\right\}+M^d(2d+2)+12d\right)\\
  &\leq 29 \binom{d+q}{d} d^2 (q+1) M^d.
  \end{align*}
  By successively applying $f_{id}$ of \autoref{le1} to the output of the network $t$, we can easily enlarge the number of hidden layers, such that the assertion also holds for $L > 8+\lceil\log_2(\max\{d, q+1\})\rceil$. Here we use that $f_{id}$ satisfies 
  \begin{align*}
  \left|f_{id}^s(x) - x\right| \leq \sum_{k=1}^s \left|f_{id}^k(x) - f_{id}^{k-1}(x)\right| = \sum_{k=1}^s \left|f_{id}(f_{id}^{k-1}(x)) - f_{id}^{k-1}(x)\right| \leq \frac{s}{M^{2p}}
  \end{align*}
  for $s \in \N$ and $x \in \left[-2\max\left\{a, \|m\|_{\infty, [-a,a]^d}\right\}, 2\max\left\{a, \|m\|_{\infty, [-a,a]^d}\right\}\right]$, where we choose
  \begin{align*}
  R \geq (s-1)8\|\sigma''\|_{\infty} \max\left\{a, \|m\|_{\infty, [-a,a]^d}\right\}^2 M^{2p}
  \end{align*}
  in \autoref{le1}. 
%  in \autoref{le1} and that $t$ satisfies
  Since $t$ satisfies
 \begin{align*}
 \|t\|_{\infty, [-a,a]^d} \leq \|t-m\|_{\infty, [-a,a]^d}+ \|m\|_{\infty, [-a,a]^b} \leq 2\max\left\{a, \|m\|_{C^q([-a,a]^d)}\right\}, 
 \end{align*}
where we use that $M^{2p} \geq c_{13}a^{5q+3}$, we can conclude that
\begin{align*}
\left|f^s(t(\x)) - m(\x)\right| \leq \left|f^{s}(t(\x)) - t(\x)\right| + \left|t(\x) - m(\x)\right| \leq c_{13} a^{5q+3} M^{-2p}
\end{align*}
holds for $\x \in [-a,a]^d$ and $s \in \N$.
  \end{proof}

\begin{proof}[\rm{\textbf{Proof of \autoref{th3}}}]
The proof is divided into \textit{two} steps.\\
\textit{Step 1: Network architecture:}  The computation of the function $m(\x)=h_1^{(\ell)}(\x)$ can be recursively described as in \eqref{hji} and \eqref{hj1}. The basic idea of the proof is to define a composed network, which approximately computes the functions $h_1^{(1)}, \dots, h_{\tilde{N}_1}^{(1)}, h_1^{(2)}, \dots, h_{\tilde{N}_2}^{(2)}, \dots, h_1^{(\ell)}$. For the approximation of $g_j^{(i)}$ we will use the networks 
  \begin{align*}
  f_{net, g_j^{(i)}} \in \mathcal{F}(L_0, r_j^{(i)}, \alpha_0)
  \end{align*}
  described in \autoref{le2}, where
  \begin{align*}
  L_0=8+\lceil\log_2(\max\{K_{\max}, p_{\max}+1\})\rceil,
  \end{align*}
  \begin{align*}
  r_j^{(i)} = 29 \binom{K_j^{(i)}+q_j^{(i)}}{q_j^{(i)}} (K_j^{(i)})^2 (q_{j}^{(i)}+1) M_{j,i}^{K_j^{(i)}}
  \end{align*}
  and
  \begin{align*}
  \alpha_0 = c_{8} a^{12} e^{6\times 2^{2(K_{\max}+1)+1}aK_{\max}}\max_{j,i} M_{j,i}^{10p_{\max}+2K_{\max}+10}
  \end{align*}

 To compute the values of $h_1^{(1)}, \dots, h_{\tilde{N}_1}^{(1)}$ we use the networks 
 \begin{align*}
 \hat{h}_1^{(1)}(\bold{x})&=   f_{net, g_{1}^{(1)}}\left(x^{(\pi(1))}, \dots, x^{(\pi(K_1^{(1)}))}\right)\\
 & \quad \vdots\\
 \hat{h}_{\tilde{N}_1}^{(1)}(\bold{x})& =  f_{net, g_{\tilde{N}_1}^{(1)}}\left(x^{(\pi(\sum_{t=1}^{\tilde{N}_1-1} K_t^{(1)} +1))}, \dots, x^{(\pi(\sum_{t=1}^{\tilde{N}_1} K_t^{(1)}))}\right).
 \end{align*}
 
 To compute the values of $h_1^{(i)}, \dots, h_{\tilde{N}_i}^{(i)}$ $(i \in \{2, \dots, \ell\})$ we use the networks
 \begin{align*}
 \hat{h}_j^{(i)}(\bold{x}) = f_{net, g_{j}^{(i)}}\left(\hat{h}_{\sum_{t=1}^{j-1} K_t^{(i)}+1}^{(i-1)}(\bold{x}), \dots, \hat{h}_{\sum_{t=1}^{j} K_t^{(i)}}^{(i-1)}(\bold{x})\right)
 \end{align*}
 for $j \in \{1, \dots, \tilde{N}_i\}$. Finally we set
 \begin{align*}
 t(\bold{x}) = \hat{h}_1^{(\ell)}(\bold{x}).
 \end{align*}

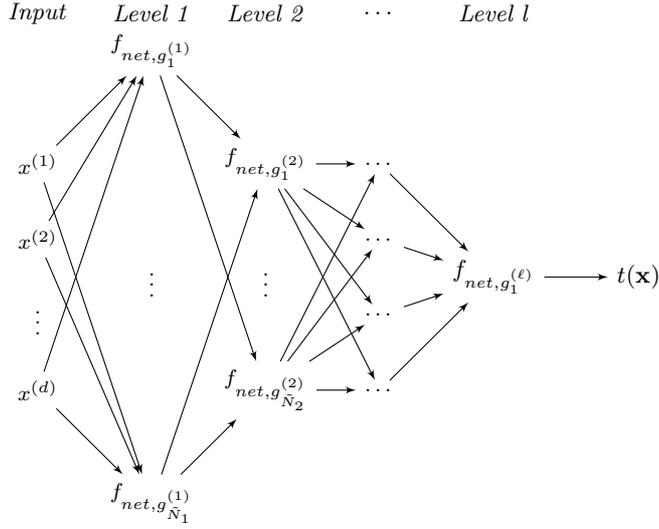
\begin{figure}
\centering
\tikzstyle{line} = [draw, -latex']
 \tikzstyle{annot} = [text width=4em, text centered]
 \tikzstyle{mycirc} = [circle,fill=white, minimum size=0.005cm]
\footnotesize{
\begin{tikzpicture}[node distance = 3cm, auto]
    % Place nodes
    
    \node [] (x1) {\scriptsize $x^{(1)}$};
    \node [below of=x1, node distance =1cm] (x2) {\scriptsize $x^{(2)}$};
    \node [below of=x2, node distance = 1cm] (dots) {\scriptsize $\vdots$};
    \node [below of=dots, node distance = 1cm] (xd) {\scriptsize $x^{(d)}$};
    \node [annot, above of=x1, node distance=2cm] (text) {\textit{Input}};
    \node [right of=x1, above of=x1, node distance=1.5cm] (fnet1) {\scriptsize $f_{net, g_1^{(1)}}$};
      \node [below of=fnet1, node distance=3cm] (fnet2) {\scriptsize $\vdots$};
        \node [right of = xd, below of =xd,  node distance=1.5cm] (fnetg1) {\scriptsize $f_{net, g_{\tilde{N}_1}^{(1)}}$};
           \node [annot, above of=fnet1, node distance=0.5cm] (text) {\textit{Level 1}};
    
     \node [right of=fnet1, below of = fnet1, node distance=1.5cm] (fnet4) {\scriptsize $f_{net, g_1^{(2)}}$};
     \node [below of=fnet4, node distance=1.5cm] (fnet2) {\scriptsize $\vdots$};
       \node [right of=fnetg1, above of= fnetg1, node distance=1.5cm] (fnet5) {\scriptsize $f_{net, g_{\tilde{N}_2}^{(2)}}$};
       \node [annot, above of=fnet4, node distance=2cm] (text) {\textit{Level 2}};
       
        \node [right of=fnet4, node distance=1.5cm] (dots1) {\scriptsize $\dots$};
              \node [below of=dots1, node distance=1cm] (dots3) {\scriptsize $\dots$}; 
                            \node [below of=dots3, node distance=1cm] (dots4) {\scriptsize $\dots$}; 
         \node [right of=fnet5, node distance=1.5cm] (dots2) {\scriptsize $\dots$};
          \node [annot, above of=dots1, node distance=2cm] (text) {$\dots$};

         \node [right of=dots1, below of=dots1,  node distance=1.5cm] (fnet6) {$f_{net, g_{1}^{(\ell)}}$};
         \node [right of= fnet6, node distance=1.9cm] (t1) {$t(\bold{x})$};
         \node [annot, above of=fnet6, node distance=3.5cm] (text) {\textit{Level l}};
      %    % Draw edges
    \path [line] (x1) -- (fnet1);
    \path [line] (x2) -- (fnet1);
    \path [line] (xd) -- (fnet1);
    
    \path [line] (x1) -- (fnetg1);
    \path [line] (x2) -- (fnetg1);
    \path [line] (xd) -- (fnetg1);
    
    \path [line] (fnet6) -- (t1);
    
     \path [line] (fnet1) -- (fnet4);
     \path [line] (fnetg1) -- (fnet4);
        \path [line] (fnet1) -- (fnet5);
     \path [line] (fnetg1) -- (fnet5);
     
        \path [line] (fnet4) -- (dots1);  
            \path [line] (fnet4) -- (dots3);  
                \path [line] (fnet4) -- (dots2);
                    \path [line] (fnet4) -- (dots4);    
          \path [line] (fnet5) -- (dots1);  
            \path [line] (fnet5) -- (dots3);  
                \path [line] (fnet5) -- (dots2);
                    \path [line] (fnet5) -- (dots4);    
         
      \path [line] (dots1) -- (fnet6);
      \path [line] (dots2) -- (fnet6);    
      \path [line] (dots3) -- (fnet6);
      \path [line] (dots4) -- (fnet6);
      
\end{tikzpicture}}
\caption{Illustration of the DNN $t(\bold{x})$}
\label{h1}
\end{figure}

\autoref{h1} illustrates the computation of the network $t(\bold{x})$. It is easy to see that $t(\bold{x})$ forms a composed network, where the networks $\hat{h}_1^{(i)}, \dots, \hat{h}_{\tilde{N}_i}^{(i)}$ are computed in parallel (i.e., in the same layers) for $i \in \{1, \dots, \ell\}$, respectively. Since each $\hat{h}_j^{(i)}$ $(j \in \{1, \dots, \tilde{N}_i\})$ needs $L_0$ layers, $r_j^{(i)}$ neurons per layer and has $\alpha_0$ as bound for its weights, 
% and we have
% \begin{align*}
% \tilde{N}_1 > \tilde{N}_2 > \dots > \tilde{N}_{\bl}
% \end{align*}
 this network is contained in the class
 \begin{align*}
   \mathcal{F}\left(\ell L_0, \max_{i \in \{1, \dots, \ell\}} \sum_{j=1}^{\tilde{N}_i} r_j^{(i)}, \alpha_0^2\right)
   \subseteq
   \mathcal{F}\left( L,r, \alpha\right).
 \end{align*}
 \textit{Step 2: Approximation error:} We define
 \begin{align*}
 g_{\max} := \max\left\{\max_{\substack{i \in \{1, \dots, \ell\}, \\  j \in \{1, \dots, \tilde{N}_i\}}} \|g_j^{(i)}\|_{\infty}, 1\right\}.
 \end{align*}
 Since each $g_j^{(i)}$ satisfies the assumption of \autoref{le2}, we can conclude that
 \begin{align}
 \label{apprg}
 \left|f_{net, g_j^{(i)}}(\x) - g_j^{(i)}(\x)\right| \leq c_{14} a^{5p_{\max}+3} \max_{j,i} M_{j,i}^{-2p_j^{(i)}}
 \end{align}
for $\x \in [-2 \max\{g_{\max}, a\}, 2\max\{g_{\max}, a\}]^{K_j^{(i)}}$, where
\begin{align*}
c_{14} = c_{13}(2g_{max})^{5p_{\max}+3}.
\end{align*} 
We show by induction that
\begin{align}
\label{appind}
\left|\hat{h}_j^{(i)}(\x) - h_j^{(i)}(\x)\right| \leq c_{14} i (K_{\max} C_{Lip})^{i-1}a^{5p_{\max}+3} \max_{j,i} M_{j,i}^{-2p_j^{(i)}}.
\end{align}
By \eqref{apprg} we can conclude that
\begin{align*}
\left|\hat{h}_j^{(1)}(\x) - h_j^{(1)}(\x)\right| \leq c_{14} 1 (K_{\max} C_{Lip})^{1-1}a^{5p_{\max}+3} \max_{j,i} M_{j,i}^{-2p_j^{(i)}}
\end{align*}
for $j \in \{1, \dots, \tilde{N}_1\}$. 
%For $\min_{j,i} M_{j,i}^{2p_j^{(i)}} \geq c_{14} a^{5p_{\max}+3}$ we can further bound the value of the network by 
%\begin{align*}
%\left|\hat{h}_j^{(1)}(\x)\right| \leq \left|\hat{h}_j^{(1)}(\x) - h_j^{(1)}(\x)\right| + g_{\max} \leq 2g_{\max}. 
%\end{align*}
 Thus we have shown that \eqref{appind} holds for $i=1$. Assume now that \eqref{appind} holds for some $i-1$ and every $j \in \{1, \dots, \tilde{N}_{i-1}\}$. Then 
\begin{align*}
 \left|\hat{h}_{j}^{(i-1)}(\x)\right| \leq \left|\hat{h}_j^{(i-1)}(\x) - h_j^{(i-1)}(\x)\right| + g_{\max} \leq 2g_{\max}
 \end{align*}
 follows directly by the induction hypothesis, where we use that
 \begin{align*}
 \min_{j,i} M_{j,i}^{2p_j^{(i)}} \geq c_{14} (i-1)(K_{\max}C_{Lip})^{i-1} a^{5p_{\max}+3}.
 \end{align*}
 Using \eqref{apprg} and the Lipschitz continuity of $g_j^{(i)}$ we can conclude that
 \begin{align*}
 &\left|\hat{h}_{j}^{(i)}(\bold{x}) - h_{j}^{(i)}(\bold{x})\right|\\
 & = \left|f_{net, g_j^{(i)}}\left(\hat{h}_{\sum_{t=1}^{j-1} K_t^{(i)} +1}^{(i-1)}, \dots, \hat{h}_{\sum_{t=1}^{j} K_t^{(i)}}^{(i-1)}\right) - g_j^{(i)}\left(\hat{h}_{\sum_{t=1}^{j-1} K_t^{(i)} +1}^{(i-1)}, \dots, \hat{h}_{\sum_{t=1}^{j} K_t^{(i)}}^{(i-1)}\right)\right|\\
 & \quad + \left|g_j^{(i)}\left(\hat{h}_{\sum_{t=1}^{j-1} K_t^{(i)} +1}^{(i-1)}, \dots, \hat{h}_{\sum_{t=1}^{j} K_t^{(i)}}^{(i-1)}\right) - g_j^{(i)}\left(h_{\sum_{t=1}^{j-1} K_t^{(i)} +1}^{(i-1)}(x), \dots, h_{\sum_{t=1}^{j} K_t^{(i)}}^{(i-1)}(x)\right)\right|\\
& \leq c_{14} a^{5p_{\max}+3} \max_{j,i} M_{j,i}^{-2p_j^{(i)}} + K_j^{(i)}  C_{Lip}
  c_{14} (i-1)(K_{max} C_{Lip})^{i-2}a^{5p_{\max}+3} \max_{j,i} M_{j,i}^{-2p_j^{(i)}}\\
  & \leq
c_{15}i (K_{max}C_{Lip})^{i-1} a^{5p_{\max}+3}\max_{j,i} M_{j,i}^{-2p_j^{(i)}}.
 \end{align*}
Thus we have shown that there exists a network $t$ satisfying 
\begin{align*}
\|t-m\|_{\infty, [-a,a]^d} \leq c_{9} a^{5p_{\max}+3} \max_{j,i}M_{j,i}^{-2p_{j}^{(i)}}.
\end{align*}
This proves the assertion of the theorem.
\end{proof}
\section{Proof of the main result}
\label{se4}
\subsection{An auxilary result from the empirical process theory}
In the proof of \autoref{th1} we use the following bound on the expected $L_2$ error of the least squares estimators. 
\begin{lemma}
\label{le9}
 Assume that the distribution of $(\bold{X},Y)$ satisfies $\EXP\{\exp(c_{16} Y^2)\} < \infty$
for some constant $c_{16} > 0$ and that the regression function $m$ is bounded in absolute value. Let $\tilde{m}_n$ be the least squares estimator 
\begin{align*}
\tilde{m}_n(\cdot) = \arg \min_{f \in \mathcal{F}_n} \frac{1}{n} \sum_{i=1}^n |Y_i - f(\bold{X}_i)|^2
\end{align*}
based on some function space $\mathcal{F}_n$ and set $m_n = T_{c_{17}  \ln n} \tilde{m}_n$ for some constant $c_{17} > 0$. Then $m_n$ satisfies
\begin{align*}
 \EXP \int |m_n(\x) - m(\x)|^2 {\PROB }_{\bold{X}} (d\x) \leq &\frac{c_{18} (\ln n)^2 \left(\ln\left(
\mathcal{N} \left(\frac{1}{n c_{17}  \ln n}, \mathcal{F}_n, \| \cdot \|_{\infty,\rm{supp}(X)}\right)
\right)+1\right)}{n}\\
&+ 2  \inf_{f \in \mathcal{F}_n} \int |f(\x)-m(\x)|^2 {\PROB}_{\bold{X}} (d\x)
\end{align*}
for $n > 1$ and some constant $c_{18} > 0$, which does not depend on $n$ or the parameters of the estimator.
\end{lemma} 
\begin{proof}[\rm{\textbf{Proof.}}]
This proof follows in a straightforward way from the proof of Theorem 1 in \cite{BaClKo09}. A complete proof can be found in the supplement of \cite{BK17}.
\end{proof}
\subsection{A bound on the covering number}
If the function class $\mathcal{F}_n$ in \autoref{le9} forms a class of fully connected DNNs $\mathcal{F}(L,r,\alpha)$ with $\alpha$ and $L$ bounded, the following result will help to bound the covering number:
\begin{lemma} \label{le3}
Let $\epsilon \geq 1/n^{c_{19}}$ and let $\mathcal{F}(L,r,\alpha)$ defined as in \eqref{F} with $\sigma: \R \to [0,1]$ Lipschitz continuous with Lipschitz constant $C_{Lip} >0$, $1 \leq \max\{a, \alpha\} \leq n^{c_{20}}$ and $L \leq c_{21}$ for large $n$ and certain constants $c_{19}, c_{20}, c_{21} > 0$. Then 
\begin{align*}
\left(\ln \mathcal{N}(\epsilon, \mathcal{F}(L,r,\alpha), \Vert \cdot \Vert_{\infty, [-a,a]^d})\right) \leq c_{22}  (1+\ln n+\ln r)  (r+1)^2
\end{align*}
holds for sufficiently large $n$ and a constant $c_{22} > 0$ independent of $n$.
\end{lemma}
\begin{proof}[\rm{\textbf{Proof.}}]
Let
\[
f(\x) =\sum_{i=1}^{r} c_{1,i}^{(L)}  f_i^{(L)}(\x) + c_{1,0}^{(L)},
\quad
\bar{f}(\x) = \sum_{i=1}^{r} \bar{c}_{1,i}^{(L)}  \bar{f}_i^{(L)}(\x) + \bar{c}_{1,0}^{(L)},
\]
for some $c_{1,0}^{(L)}, \bar{c}_{1,0}^{(L)}, \dots, c_{1,r}^{(L)} , \bar{c}_{1,r}^{(L)} \in \mathbb{R}$ and for $f_i^{(L)}$'s, $\bar{f}_i^{(L)}$'s recursively defined by
\begin{eqnarray*}
&&
f_i^{(s)}(\x)= \sigma\left(\sum_{j=1}^{r} c_{i,j}^{(s-1)} 
  f_j^{(s-1)}(\x) + c_{i,0}^{(s-1)} \right),
\\
&&
\bar{f}_i^{(s)}(\x)= \sigma\left(\sum_{j=1}^{r} \bar{c}_{i,j}^{(s-1)}  \bar{f}_j^{(s-1)}(\x) + \bar{c}_{i,0}^{(s-1)} \right)
\end{eqnarray*}
for some $c_{i,0}^{(s-1)}, \bar{c}_{i,0}^{(s-1)}, \dots, c_{i, r}^{(s-1)} , \bar{c}_{i, r}^{(s-1)} \in \mathbb{R}$,
$s \in \{2, \dots, L\}$,
and
\[
f_i^{(1)}(\x) = \sigma \left(\sum_{j=1}^d c_{i,j}^{(0)}  x^{(j)} +
  c_{i,0}^{(0)}\right), \quad
\bar{f}_i^{(1)}(\x) = \sigma \left(\sum_{j=1}^d \bar{c}_{i,j}^{(0)}  x^{(j)} + \bar{c}_{i,0}^{(0)} \right)
\]
for some $c_{i,0}^{(0)}, \bar{c}_{i,0}^{(0)}, \dots, \bar{c}_{i,d}^{(0)} \in \mathbb{R}$.
Let $C_{Lip} \geq 1$ be an upper bound on the Lipschitz
constant of $\sigma$.
Then
\begin{eqnarray*}
|f(\x)-\bar{f}(\x)|
&\leq &
\sum_{i=1}^{r} |c_{1,i}^{(L)}| | f_i^{(L)}(\x)-
\bar{f}_i^{(L)}(\x)|+ |c_{1,0}^{(L)}- \bar{c}_{1,0}^{(L)}|+
\sum_{i=1}^{r} |c_{1,i}^{(L)} - \bar{c}_{1,i}^{(L)}|
   |\bar{f}_i^{(L)}(\x)|
\\
&
\leq
&
r  \max_{i\in \{1,\dots,r\}}  |c_{1,i}^{(L)}|
\max_{i \in \{1,\dots,r\}}
| f_i^{(L)}(\x)-
\bar{f}_i^{(L)}(\x)|+ |c_{1,0}^{(L)}- \bar{c}_{1,0}^{(L)}|
\\
& &\quad 
+
r 
\max_{i \in \{1, \dots,r\}} |c_i^{(L)} - \bar{c}_i^{(L)}|,
\end{eqnarray*}
\begin{eqnarray*}
|f_i^{(s)}(x)-\bar{f}_i^{(s)}(\x)| &\leq &
C_{Lip} 
\left|
\sum_{j=1}^{r} c_{i,j}^{(s-1)} 
  f_j^{(s-1)}(\x) + c_{i,0}^{(s-1)}
-
\left(\sum_{j=1}^{r} \bar{c}_{i,j}^{(s-1)}  \bar{ f}_j^{(s-1)}(\x) + \bar{c}_{i,0}^{(s-1)} \right)
\right|\\
&
\leq
&
C_{Lip}  r 
\max_{j \in \{1,\dots,r\}} |c_{i,j}^{(s-1)}| 
\max_{j \in \{1,\dots,r\}} |  f_j^{(s-1)}(\x) -  \bar{f}_j^{(s-1)}(\x)|
\\
&&+
C_{Lip}  r 
\max_{j \in \{1,\dots,r\}} |c_{i,j}^{(s-1)}- \bar{c}_{i,j}^{(s-1)}|
+
C_{Lip}  |c_{i,0}^{(s-1)}-  \bar{c}_{i,0}^{(s-1)}|
\end{eqnarray*}
for $s \in \{2, \dots, L\}$ and
\[
|f_i^{(1)}(\x)-\bar{f}_i^{(1)}(\x)|
\leq
C_{Lip}
(d+1)
\max_{j\in \{0,\dots,d\}}|c_{i,j}^{(0)}-\bar{c}_{i,j}^{(0)}|  a.
\]
In the sequel we will use the abbreviation
\[
\max_{i,j,s}|c_{i,j}^{(s)}|
=
\max\{
\max_{i} |c_{1,i}^{(L)}|,
\max_{i,j,s}|c_{i,j}^{(s)}|
\}.
\]
Recursively we conclude
\begin{eqnarray*}
|f(\x)-\bar{f}(\x)| &\leq &
(r+1) \max_{i\in \{0, \dots,r\}} |c_{1,i}^{(L)} - \bar{c}_{1,i}^{(L)}|
\\
&&
\quad
+
 r  \max_{i,j,s} |c_{i,j}^{(s)}| 
C_{Lip} (r+1) 
\max_{i,j} |c_{i,j}^{(L-1)}- \bar{c}_{i,j}^{(L-1)}| \\
&&
\quad
+
 r  (\max_{i,j,s} |c_{i,j}^{(s)}| 
C_{Lip} (r+1))^2 
\max_{i,j} |c_{i,j}^{(L-2)}- \bar{c}_{i,j}^{(L-2)}|
+
\dots
\\
&&
\quad
+
 r  (\max_{i,j,s} |c_{i,j}^{(s)}| 
C_{Lip} (r+1))^{L-1} 
\max_{i,j} |c_{i,j}^{(1)}- \bar{c}_{i,j}^{(1)}|
\\
&&
\quad
+
 r  (\max_{i,j,s} |c_{i,j}^{(s)}| 
 C_{Lip} (r+1))^{L-1}  C_{Lip}
  (d+1)  a 
 \max_{i,j} |c_{i,j}^{(0)}-\bar{c}_{i,j}^{(0)}|.
\end{eqnarray*}
Provided we have
\[
\max_{i\in \{0,\dots,r\}}
|c_{1,i}^{(L)}-\bar{c}_{1,i}^{(L)}|
\leq
\frac{\epsilon}{(L+1)  (r+1)},
\]
\[
\max_{i,j} |c_{i,j}^{(t)}- \bar{c}_{i,j}^{(t)}|
\leq \frac{\epsilon}{
(L+1) 
r  (\max_{i,j,s} |c_{i,j}^{(s)}| 
C_{Lip} (r+1))^{L-t}
}
\]
for $t\in \{1,\dots,L\}$
and
\begin{align*}
\max_{i,j} |c_{i,j}^{(0)}-\bar{c}_{i,j}^{(0)}| \leq  \frac{\epsilon}{
(L+1)  r  (\max_{i,j,s} |c_{i,j}^{(s)}| 
 C_{Lip} (r+1))^{L-1}  C_{Lip}
  (d+1)  a
 }
\end{align*}
this implies
\[
|f(\x)-\bar{f}(\x)| \leq \underbrace{\frac{\epsilon}{L+1} + \frac{\epsilon}{L+1} + \dots + \frac{\epsilon}{L+1}}_{L+1-\text{times}} = 
 \epsilon.
\]
Assume that $c_{i,j}^{(s)}$ are all contained in the interval $[-\alpha,\alpha]$,
where $\alpha \geq 1$. By discretizing this interval on the various levels
$s$ for each of the at most $(r+1)^2$ weights used in this level
accordingly, we see that we can construct a supremum norm cover of
size
\begin{align*}
&
\prod_{t=0}^{L-1}
\left(
\frac{ 2 \alpha 
(L+1) 
(r+1)  (\alpha 
C_{Lip} ( r+1))^{t}
}{\epsilon}
\right)^{(r+1)^2}\\
& \quad \quad \cdot
\frac{2 \alpha 
(L+1)  (r+1)  (\alpha 
 C_{Lip} (r+1))^{L-1}  C_{Lip}
  (d+1)  a
 }{\epsilon}
\\
& \leq
\left(
\frac{ 2 \alpha 
(L+1) 
(r+1)  (\alpha 
C_{Lip} ( r+1))^{L-1}
}{\epsilon}
\right)^{L(r+1)^2}
\frac{2
(L+1)(\alpha 
 C_{Lip} (r+1))^{L} 
  (d+1)  a
 }{\epsilon}
\\
&
\leq
c_{28}
\left(
\frac{ 
(L+1) 
(\alpha 
  C_{Lip} ( r+1))^{L}
(d+1)  a
}{\epsilon}
\right)^{(L+1)(r+1)^2}.
\end{align*}
\end{proof}
%\subsection{Proof of Theorem \ref{th1}}
\subsection{Proof of \autoref{th1}}
%Combining the result presented in Section \ref{se3} with \autoref{le9} and \autoref{le3}, we are able, to prove our main result as follows:
%\begin{proof}[\rm{\textbf{Proof of \autoref{th1}.}}]
Let $a_n =( \ln n)^{3/(2 \times (5p_{\max}+3))}$. For $n$ sufficiently large the relation $\rm{supp}(\bold{X}) \subseteq$ $[-a_n,a_n]^d$ holds, which implies $\mathcal{N}(\delta, \mathcal{G}, \Vert \cdot \Vert_{\infty, \rm{supp}(\bold{X})}) \leq \mathcal{N}(\delta, \mathcal{G}, \Vert \cdot \Vert_{\infty, [-a_n,a_n]^d})$ for an arbitrary function space $\mathcal{G}$ and $\delta > 0$.
%By using standard techniques in empirical process theory
%as in the proof of Theorem 1 \cite{BaClKo09}
%(cf., Lemma 1 in \cite{BK17} for a corresponding
%complete proof) it is possible to show
Application of \autoref{le9} leads to
\begin{align*}
 & \EXP \int |m_n(\x) - m(\x)|^2 {\PROB}_{\bold{X}} (d\x)\\
   &\leq \frac{c_{18} (\ln n)^2 \left(\ln\left(
\mathcal{N} \left(\frac{1}{n c_4  \ln n}, \mathcal{F}(L_n,r_n, \alpha_n), \|\cdot\|_{\infty,\rm{supp}(\bold{X})}\right)
\right)+1\right)}{n}\\
&\quad + 2  \inf_{f \in \mathcal{F}(L_n r_n, \alpha_n)} \int |f(\x)-m(\x)|^2 {\PROB}_{\bold{X}} (d\x).
\end{align*}
%for $n>1$ and some constant $c_{25}>0$, which does not depend on $n$ or the parameters of the estimate.
%Let
% \begin{align*}
% (\bar{p},\bar{K}) \in \mathcal{P} \quad \text{with} \quad (\bar{p}, \bar{K})=\arg \min_{(p,K) \in \mathcal{P}} \frac{p}{K}.
% \end{align*}
Set
\begin{align*}
  (\bar{p},\bar{K}) = \arg \max_{(p,K) \in \mathcal{P}} n^{-\frac{2p}{2p+K}}.
  \end{align*}
The fact that  $1/(n  c_4  \ln n) \geq 1/n^{c_{19}}$, $\max\{a_n, \alpha\} \leq n^{c_{20}}$ and $r_n \leq c_{21}  n^{1/2(2\bar{p}/\bar{K} + 1)}$
holds for $c_{19}, c_{20}, c_{21}>0$, 
allows us to apply \autoref{le3} to bound the first summand by 
\begin{align}
\label{th1eq1}
%%&\ln \left(\mathcal{N} \left(\frac{1}{n c_5  \ln n}, \mathcal{F}(L_n,r_n, \alpha_n), \|\cdot\|_{\infty,\rm{supp}(X)}\right)\right)\notag\\
%%\leq & c_{31}  \ln n  \max_{j,i} M_{j,i}^{K_j^{(i)}}\notag\\
%\leq & c_{31}  \ln n  n^{\frac{1}{2\bar{p}/\bar{K} + 1}}.
\frac{c_{18}(\ln n)^2   c_{22}  \left(1+\ln n + \left(\ln c_{23}  n^{\frac{1}{2(2\bar{p}/\bar{K} +1)}}\right)\right)  c_{11}  n^{\frac{1}{2\bar{p}/\bar{K} +1}}}{n} &\leq \frac{c_{24}  (\ln n)^3  n^{\frac{1}{2\bar{p}/\bar{K} +1}}}{n} \notag\\
&\leq  c_{24}  (\ln n)^3  n^{-\frac{2\bar{p}}{2\bar{p} +\bar{K}}}
\end{align}
for a sufficiently large $n$. 
Regarding the second summand we apply \autoref{th3}, where we choose $M_{j,i} = \bigg\lceil n^{1/2(2p_j^{(i)}+K_j^{(i)})}\bigg\rceil$. 
%for a sufficiently large constant $c_{33} > 0 $. 

Since 

\begin{align*}
&\max_{i \in \{1, \dots, \ell\}} \sum_{j=1}^{\tilde{N}_i} 29 \binom{K_j^{(i)}+q_j^{(i)}}{q_j^{(i)}} (K_j^{(i)})^2 (q_{j}^{(i)}+1) M_{j,i}^{K_j^{(i)}}\\
&\leq \tilde{N}_1 29 \binom{K_{\max}+p_{\max}}{p_{\max}}K_{\max}^2 p_{\max} \max_{j,i} M_{j,i}^{K_j^{(i)}} \\
& = \tilde{N}_1 29 \binom{K_{\max}+p_{\max}}{p_{\max}}K_{\max}^2 p_{\max} \max_{j,i} \bigg\lceil n^{1/(2(2p_j^{(i)}+K_j^{(i)}))}\bigg\rceil^{K_j^{(i)}}\\
& \leq \tilde{N}_1 29 \binom{K_{\max}+p_{\max}}{p_{\max}}K_{\max}^2 p_{\max} \max_{j,i} \bigg( n^{1/(2(2p_j^{(i)}+K_j^{(i)}))}+1\bigg)^{K_j^{(i)}}\\
& \leq \tilde{N}_1 29 \binom{K_{\max}+p_{\max}}{p_{\max}}K_{\max}^2 p_{\max} \max_{j,i} \bigg( 2n^{1/(2(2p_j^{(i)}+K_j^{(i)}))}\bigg)^{K_j^{(i)}}\\
& \leq 2^{K_{\max}} \tilde{N}_1 29 \binom{K_{\max}+p_{\max}}{p_{\max}}K_{\max}^2 p_{\max} \max_{j,i}  n^{K_j^{(i)}/(2(2p_j^{(i)}+K_j^{(i)}))} = r
\end{align*}
and 
\begin{align*}
\alpha&=c_{8} a_n^{24} e^{12\times 2^{2(K_{\max}+1)+1}a_nK_{\max}}\max_{j,i} M_{j,i}^{20p_{\max}+4K_{\max}+20}\\
& = c_{8} \left((\ln n)^{3/(2 \times (5p_{\max}+3))}\right)^{24} e^{12\times 2^{2(K_{\max}+1)+1}(\ln n)^{3/(2 \times (5p_{\max}+3))}K_{\max}}\\
& \hspace*{6cm} \cdot \max_{j,i} \bigg\lceil n^{1/2(2p_j^{(i)}+K_j^{(i)})}\bigg\rceil^{20p_{\max}+4K_{\max}+20} \leq n^{c_{25}}
\end{align*}
%\begin{align*}
%\alpha &= c_{9}  \ln(n)^{\frac{3(7p_{\max}+1)}{2 \times 5N}}  c_{31}  n^{\frac{1}{2\bar{p}/\bar{K}+1}}  \max_{j,i} n^{\frac{8  (p_{\max}+8( K_{\max}+1)}{2p_j^{(i)}+K_j^{(i)}}} \leq n^{c_{31}}
%\end{align*}
for $c_{25} >0$ sufficiently large,
the resulting values of $r$ and $\alpha$ are consistent with $r_n$ and $\alpha_n$ in \autoref{th1}. \autoref{th3} allows us to bound $ \inf_{f \in \mathcal{F}(L_n r_n, \alpha_n)} \int |f(\x)-m(\x)|^2 {\PROB}_{\bold{X}} (d\x)$ by
\begin{align*}
 &c_{26}  \left(a_n^{5p_{\max}+3}\right)^2  \max_{j,i} M_{j,i}^{-4p_j^{(i)}} =  c_{26}  (\ln n)^3 \max_{j,i} n^{-\frac{4p_j^{(i)}}{2(2p_j^{(i)} + K_j^{(i)})}}.
\end{align*}
This together with \eqref{th1eq1} and the fact that 
\begin{align*}
 n^{-\frac{2\bar{p}}{2\bar{p} +\bar{K}}} = \max_{j,i} n^{-\frac{2p_j^{(i)}}{2p_j^{(i)} + K_j^{(i)}}}
\end{align*}
implies the assertion.

\bibliographystyle{acm}
\bibliography{Literatur}
%\newpage
%\section*{Supplementary material for the referees}
 \end{document}